\documentclass[10pt]{amsart}
\usepackage{amsmath}

\usepackage{amsfonts}
\usepackage{amssymb}
\newcommand{\vertiii}[1]{{\left\vert\kern-0.25ex\left\vert\kern-0.25ex\left\vert #1
    \right\vert\kern-0.25ex\right\vert\kern-0.25ex\right\vert}}

\usepackage{amssymb, latexsym}
\usepackage{amssymb}
\usepackage{graphicx}
\usepackage{color}
\allowdisplaybreaks

\def\pasdegrille{\let\grille = \pasgrille}

\def\aat#1#2#3{
\divide \dimen1 by 48 \dimen3=\dimen1 \multiply \dimen1 by #1
\advance \dimen1 by -\dimen3 \divide \dimen1 by 101 \multiply
\dimen1 by 100 \divide \dimen2 by \count11 \multiply \dimen2 by #2
\setbox0=\hbox{#3}\ht0=0pt\dp0=0pt
  \rlap{\kern\dimen1 \vbox to0pt{\kern-\dimen2\box0\vss}}\dimen1= \wd1
\dimen2=\ht1}
\def\pasgrille{
\count12= \dimen1 \divide \count12 by 50 \divide \dimen2 by \count12
\count11 =\dimen2 \ \divide \dimen1 by 48
\setlength{\unitlength}{\dimen1} \smash{\rlap{\ }} \dimen1= \wd1
\dimen2=\ht1 }
\def\grille{
\count12= \dimen1 \divide \count12 by 50 \divide \dimen2 by \count12
\count11 =\dimen2 \ \divide \dimen1 by 48
\setlength{\unitlength}{\dimen1}
\smash{\rlap{\graphpaper[1](0,0)(50, \count11)}} \dimen1= \wd1
\dimen2=\ht1 }

\pasdegrille

\newcommand{\Z}{{\mathbb{Z}}}

\newcommand{\R}{{\mathbb{R}}}

\newcommand{\eps}{{\varepsilon}}

\newcommand{\g}{{ g}}
\newcommand{\titau}{\tilde{\tau}}
\newcommand{\ld}{\lambda}
\newcommand{\beq}{\begin{equation}}
\newcommand{\eeq}{\end{equation}}

\theoremstyle{plain}
\newtheorem{theorem}{Theorem}
\newtheorem{proposition}[theorem]{Proposition}
\newtheorem{lemma}[theorem]{Lemma}

\theoremstyle{definition}

\newtheorem{remark}[theorem]{Remark}

\newcounter{smalllist}

\numberwithin{equation}{section}
\numberwithin{theorem}{section}

\begin{document}

\title[Bilinear Kakeya-Nikodym]
{\Large Bilinear Kakeya-Nikodym averages of eigenfunctions on compact Riemannian surfaces}
\author{Changxing Miao}
\address{Institute of Applied Physics and Computational Mathematics\\
Beijing, 100088\\
China}\email{miao\_changxing@iapcm.ac.cn }
\author{Christopher D. Sogge}
\address{Department of Mathematics\\
Johns Hopkins University\\
Baltimore, MD 21218, USA}
\email{sogge@jhu.edu }
\author{Yakun Xi}
\address{Department of Mathematics\\
Johns Hopkins University\\
Baltimore, MD 21218, USA}
\email{ykxi@math.jhu.edu}
\author{Jianwei Yang}
\address{LAGA (UMR 7539)\\ Institut Galil\'ee\\ Universit\'e Paris 13\\ Sorbonne Paris Cit\'e\\  \^Ile-de-France.}
\address{
Beijing International Center for Mathematical Research\\ 
Peking University\\
Beijing 100871, China}
\email{geewey\_young@pku.edu.cn}
\thanks{The first author was supported in part  by the NSF of China. The second author was supported  in part by the NSF grants DMS-1361476. The fourth author was supported by ERC Advanced Grant No. 291214BLOWDISOL}
\begin{abstract}
We obtain an improvement of the bilinear estimates of Burq, G\'erard and Tzvetkov~\cite{BurqGerardTzvetkov:Invention}
in the spirit of the refined Kakeya-Nikodym estimates \cite{Blair Sogge:refine} of Blair and the second author.  We do this
by using microlocal techniques and a bilinear version of H\"ormander's oscillatory integral theorem in \cite{Hormander:FLp}.
\end{abstract}

\maketitle


%
%

\section{Introduction}
Let $(M,\g)$ be a two-dimensional
compact boundaryless Riemannian manifold
with Laplacian $\Delta_\g$.
If $e_\lambda$ are the associated eigenfunctions of $\sqrt{-\Delta_\g}$
such that $-\Delta_\g e_\lambda=\lambda^2 e_\lambda$, then it is well known that
\begin{equation}
  \label{eq:Sogge classical estimate}
  \|e_\lambda\|_{L^4(M)}\leq C\,\lambda^{\frac18}\,\|e_\lambda\|_{L^2(M)},
\end{equation}
which was proved in \cite{Cef} using
approximate spectral projectors
$\chi_\ld=\chi(\lambda-\sqrt{-\Delta_\g})$ and showing
\begin{equation}
  \label{eq:spectral est}
  \|\chi_\lambda \,f\|_{L^4(M)}\le C\,\lambda^{\frac18}\,\|f\|_{L^2(M)}.
\end{equation}

If $0<\lambda\le \mu$ and
$e_\lambda,e_\mu$ are two associated
eigenfunctions of  $\sqrt{-\Delta_\g}$ as above,
Burq et al \cite{BurqGerardTzvetkov:Invention} proved the following bilinear $L^2$-refinement
of \eqref{eq:Sogge classical estimate}
\begin{equation}\label{eq:BGT estimate}
\|e_\lambda e_\mu\|_{L^2(M)}
\le C\,\lambda^\frac14\, \|e_\lambda\|_{L^2(M)}\|e_\mu\|_{L^2(M)},
\end{equation}
as a consequence of a more general bilinear estimate on the reproducing operators
\begin{equation}\label{eq:BGT estimate-1}
\bigl\|\chi_\lambda f\chi_\mu g\bigr\|_{L^2(M)}
\le C\,\ld^\frac14\,\|f\|_{L^2(M)}\|g\|_{L^2(M)}.
\end{equation}
The bilinear estimate \eqref{eq:BGT estimate} plays an important role
in the theory of nonlinear Schr\"odinger equations on compact Riemannian surfaces
and it is sharp
in the case when $M=\mathbb S^2$ endowed with the canonical metric
and $e_\lambda(x)=h_p(x)$, $e_\mu(x)=h_q(x)$ are highest weight
spherical harmonic functions of degree $p$ and $q$, concentrating along the equator
\[\bigl\{x=(x_1,x_2,x_3):x_1^2+x_2^2=1,x_3=0\bigr\}\]
 with $\lambda^2=p(p+1)$, $\mu^2=q(q+1)$.
Indeed, one may take
$h_k(x)=(x_1+ix_2)^k$
to see $\|h_k\|_2\approx k^{-1/4}$ by direct computation.

\medskip

In Section \ref{sect:generic failure},
we will construct a generic example to show
the optimality of \eqref{eq:BGT estimate-1}
and exhibit that the mechanism responsible for the optimality
seems to be the existence
of eigenfunctions concentrating along a tubular neighborhood
of a segment of a geodesic.
As observed in \cite{Sogge:93book}, \eqref{eq:spectral est}
is saturated by constructing
an oscillatory integral which highly concentrates along a geodesic.
The dynamical behavior of geodesic flows on $M$ accounts for
the analytical properties of eigenfunctions
exhibits the transference of mathematical
theory from classical mechanics to quantum mechanics
(see \cite{Sogge:hangzhou-book}).
\medskip

That the eigenfunctions concentrating
along geodesics yield sharp spectral projector
inequalities leads naturally to
the refinement of \eqref{eq:Sogge classical estimate} in \cite{Sogge:Tohoku},
where it is proved for an $L^2$ normalized eigenfunction $e_\ld$,
its $L^4$-norm is essentially
bounded by a power of
\begin{equation}
\sup_{\gamma\in\it\Pi}\frac{1}{|\,T_{\ld^{-1/2}}(\gamma)\,|}
\int_{T_{\ld^{-1/2}}(\gamma)}|\,e_\ld(x)\,|^2 \, dx,
\end{equation}
where  $\it\Pi$ denotes the collection of all unit geodesics
and  $T_\delta(\gamma)$ is a tubular $\delta$-neighborhood about the geodesic
$\gamma$. This fact motivates the Kakeya-Nikodym maximal average
phenomena measuring the size and concentration of eigenfunctions.

This result was refined by Blair and Sogge \cite{Blair Sogge:refine}, where the authors proved
for every $0<\eps\leq 1/2$, there is a $C=C(\eps,M)$ so that
\begin{equation}
  \label{eq:refined}
  \|e_\ld\|_{L^4(M)}\leq C \,\ld^\frac18\,\|e_\ld\|^\frac12_{L^2(M)}\times
  \biggl(\sup_{\gamma\in\it\Pi}
  \int_{T_{\ld^{-\frac12+\eps}}(\gamma)}|e_\ld(x)|^2dx\biggr)^\frac14.
\end{equation}
We shall assume throughout that our eigenfunctions are $L^2$-normalized, but we shall
formulate our main estimates as in \eqref{eq:refined} to emphasize the difference between
the norms over all of $M$ and over shrinking tubes.

As mentioned in \cite{Blair Sogge:refine} it would be interesting to see whether the $\eps$-loss in \eqref{eq:refined} can be eliminated.
Further results for higher dimensions are in \cite{BSRef}.

Inspired by \cite{Sogge:Tohoku}, we are
interested in the bilinear version of the main result in \cite{Sogge:Tohoku},
namely, searching for the essentially appropriate control of $\|e_\lambda e_\mu\|_2$
by means of Kakeya-Nikodym maximal averages.
In fact, we will obtain a better result by establishing the microlocal
version of Kakeya-Nikodym average in the spirit of \cite{Blair Sogge:refine},
and our main result reads
\begin{theorem}
  \label{theorem:main}
  Assume $0<\ld\leq \mu$ and $e_\ld$, $e_\mu$ are two eigenfunctions of $\sqrt{-\Delta_{\g}}$
  associated to the frequencies $\ld $ and $\mu$ respectively.
  Then for every $0<\eps\leq \frac12$, we have a $C_{\eps}>0$ such that
  \begin{equation}
    \label{eq:main result1}
    \|e_\ld e_\mu\|_{L^2(M)}\leq C_{\eps}\ld^{\frac{\eps}{2}}\|e_\mu\|_{L^2(M)} \vertiii{e_\ld}_{KN(\ld,\eps)},
  \end{equation}
  and
    \begin{equation}
    \label{eq:main result2}
    \|e_\ld e_\mu\|_{L^2(M)}\leq C_{\eps}\ld^{\frac{\eps}{2}}\|e_\ld\|_{L^2(M)} \vertiii{e_\mu}_{KN(\ld,\eps)},
  \end{equation}
  where the Kakeya-Nikodym norm is defined by
  \begin{equation}
    \label{eq:define the KN}
    \vertiii{f} _{KN(\ld,\eps)}=
    \biggl(\sup_{\gamma\in\it\Pi}\ld^{\frac12-\eps}\int_{ T_{\ld^{-\frac12+\eps}}(\gamma)}|f(x)|^2\,dx\biggr)^\frac12.
  \end{equation}
\end{theorem}

Note also that we can reformulate our main estimates as follows
\begin{equation}\tag{1.7$'$}\label{1.7'}
\|e_\lambda e_\mu\|_{L^2(M)} \le C_\varepsilon \lambda^{\frac14} \|e_\mu\|_{L^2(M)}\times
\Bigl( \, \sup_{\gamma\in \varPi}  \int_{T_{\ld^{-\frac12+\eps}}(\gamma)} |e_\lambda|^2 \, dx\Bigr)^{\frac12},
\end{equation}
and
\begin{equation} \label{1.8'} \tag{1.8$'$}
\|e_\lambda e_\mu\|_{L^2(M)} \le C_\varepsilon \lambda^{\frac14} \|e_\lambda\|_{L^2(M)}\times
\Bigl(\, \sup_{\gamma\in \varPi}  \int_{T_{\ld^{-\frac12+\eps}}(\gamma)} |e_\mu|^2 \, dx\Bigr)^{\frac12},
\end{equation}
both of which are bilinear variants of \eqref{eq:refined}.
Also, by taking the geometric means of \eqref{eq:main result1} and \eqref{eq:main result2} one of course has that
 \begin{equation}
   \label{eq:main result3}
   \|e_\ld e_\mu\|_{L^2(M)}\leq C_{\eps}\ld^{\frac{\eps}{2}}\|e_\ld\|_{L^2(M)}^\frac12\|e_\mu\|_{L^2(M)}^\frac12 \vertiii{e_\ld}_{KN(\ld,\eps)}^\frac12 \vertiii{e_\mu}_{KN(\ld,\eps)}^\frac12.
\end{equation}
  
Note that it is the geodesic tubes corresponding to the lower frequency
that accounts for the optimal upper bound of $\|e_\ld e_\mu\|_2$.   We point out that in \eqref{eq:main result2} one cannot take
the $KN(\mu,\varepsilon)$-norm of $e_\mu$.  For on ${\mathbb T}^n\approx (-\pi,\pi]^n$ if $e_\lambda=e^{ij\cdot x}$, $|j|=\lambda$,
and $e_\mu=e^{ik\cdot x}$, $|k|=\mu$, the analog of \eqref{eq:main result2} involving $ \vertiii{e_\mu}_{KN(\mu,\eps)}$ is
obviously false for small $\varepsilon>0$  if $\mu \gg \lambda$.

Note also that if $e_\mu$ is replaced by a subsequence, $e_{\mu_{j_k}}$ of quantum ergodic eigenfunctions (see \cite{Sogge:hangzhou-book})
then \eqref{1.8'} implies that $\|e_\lambda e_{\mu_{j_k}}\|_{L^2(M)}\to 1$ as $\mu_{j_k}\to \infty$.  This is another reason why it would be
interesting to know whether the analog of \eqref{1.8'} is valid with $\varepsilon=0$ there.

This paper is organized as follows.
In Section 2, we construct an example to show
the sharpness of \eqref{eq:BGT estimate-1}.
In Section 3, we introduce some basic preliminaries
and reduce the proof of Theorem \ref{theorem:main}
to the situation, where the strategy in \cite{Blair Sogge:refine}
can be applied.
In Section 4, we employ the orthogonality argument
to conclude the theorem by assuming
a specific bilinear oscillatory
integral inequality.
Finally, we prove this inequality
in Section 5 based on the instrument in
\cite{BurqGerardTzvetkov:Invention}, which provides
a bilinear version of H\"ormander's oscillatory integral theorem \cite{Hormander:FLp}.
We shall assume $0<\ld\le \mu$
throughout this paper.


%
%

\section{A generic example}\label{sect:generic failure}
In this section, we shall  construct an example
showing 
the optimality of  the universal bounds \eqref{eq:BGT estimate-1}.
We will use  approximate spectral projectors
 $\chi_\lambda$ and $\chi_\mu$
which reproduce eigenfunctions and
can be written as  proper
Fourier integral operators
up to a smooth error.
\medskip

Without loss of generality, we may assume the injectivity radius of $M$ is sufficiently large. Take a Schwartz function $\chi\in \mathcal S(\mathbb R)$
with $\chi(0)=1$ and $\widehat \chi $ supported
in $[1,2]$, so that
the spectral projectors are represented by
$$\chi_\lambda f(x)=\lambda^{1/2}\mathcal T_\lambda f(x)+\mathcal R_\lambda f(x),\;
\chi_\mu g(x)=\mu^{1/2}\mathcal T_\mu g(x)+\mathcal R_\mu g(x),$$
where
\[\|\mathcal R_\lambda f\|_{L^\infty(M)}\leq C_N\lambda^{-N}\|f\|_{L^1(M)},\;
\|\mathcal R_\mu g\|_{L^\infty(M)}\leq C_N\mu^{-N} \|g\|_{L^1(M)},\]
for all $N=1,2,\ldots$ , and the main terms read
\begin{align}
\mathcal T_\lambda f(x)=&\int_{M}e^{i\lambda d_\g(x,y)}a(x,y,\lambda)\,f(y)\,dy,\\
\mathcal T_\mu g(x)=&\int_{M}e^{i\mu d_\g(x,z)}a(x,z,\mu)\,g(z)\,dz.
\end{align}
Here $d_\g(x,y)$ is the geodesic distance between $x,y\in M$, and
the amplitudes $a(x,y,\ld),a(x,z,\mu)\in C^\infty$
have the following property
\[|\partial^\alpha_{x,\,y}a(x,y,\lambda)|
+|\partial^\alpha_{x,\,z}a(x,z,\mu)|
\leq C_\alpha,\quad\text{for all}\;\alpha.\]
Moreover $a(x,y,\lambda)=0$ if
$d_\g(x,y)\not\in(1,2)$ and likewise for $a(x,z,\mu)$.  (See \cite[Lemma 5.1.3]{Sogge:93book}.)

\medskip

After applying a partition of unity, for small $\delta$ fixed,
we may fix three points $x_0,y_0,z_0\in M$ with $1\le d_\g(x_0,y_0)\le 2,\ 1\le d_\g(x_0,z_0)\le 2$, and assume that $a(x,y,\lambda)$ vanishes outside the region
$\{(x,y)|\,x\in B(x_0,\delta), \,y\in B(y_0,\delta)\}$, $a(x,z,\mu)$ vanishes outside the region
$\{(x,z)|\,x\in B(x_0,\delta),\,z\in B(z_0,\delta)\}$.
To see the sharpness of \eqref{eq:BGT estimate-1},
we will prove the following result.
\begin{proposition}
  \label{prop:example}
  There exist $f$ and $g$ such that for some $C>0$,
  \begin{equation}
    \label{eq:example}
    \bigl\|\mathcal T_\lambda f\,\mathcal T_\mu g\bigr\|_{L^2}
    \geq C \,\ld^{-1/4}\mu^{-1/2}\|f\|_{L^2}\|g\|_{L^2}.
  \end{equation}
\end{proposition}
We will choose suitable $f$ and $g$
concentrating along  a segment of
the geodesic $\gamma_0$ connecting $x_0$
and $y_0$ with  appropriate oscillations.
The explicit expression of $f$ and $g$ will yield automatically
upper bounds on $\|f\|_2\|g\|_2$.
On the other hand,
we will see there is a strip region
$\Omega_\mu$ containing $x_0$
such that $\|\mathcal T_\ld f\mathcal T_\mu g\|_{L^2(\Omega_\mu)}$
is bounded below by $(\ld\mu)^{-1/2}$
times the upper bound of $\ld^\frac14\|f\|_2\|g\|_2.$

\medskip

Recall first the geodesic normal coordinate centered at $y_0$.
Let $\{e_1, e_2\}$ be the orthonormal basis in $T_{y_0} M$
such that $e_1$ is the tangent vector of $\gamma_0$,
pointing to $x_0$.
The exponential map $\exp_{y_0}$ is a smooth diffeomorphism between the ball
$\{Y\in T_{y_0}M: Y=Y_1e_1+Y_2e_2, |Y|<10\}$
and $B(y_0,10)$.
Let $\{\omega_1,\omega_2\}$ be
the dual basis of $\{e_1,e_2\}$
and set $y_j=\omega_j\circ \exp_{y_0}^{-1}$ for $j=1,2$.
Then $\{y_1,y_2\}$
is the Riemannian geodesic normal coordinates
such that $y_0=0$ and
$$
\begin{cases}
\g_{ij}(0)&=\delta_{ij},\\
d \g_{ij}(0)&=0,
\end{cases}
\quad\text{ for all }\;i,j=1,2.
$$
In particular, $\Gamma_{ij}^k(0)=0,\forall\, i,j,k=1,2,$
and $d G(0)=0$ with $G={\rm det}(\g_{ij})$.
In this coordinate system,
$\gamma_0$ is parameterized by $t\mapsto \{(t,0)\}$.

\begin{lemma}
  \label{lemma:distance function}
  If we denote by $\phi(x,y)=d_\g(x,y)$, then in these coordinates
  $\phi(x,0)=|x|$. Moreover, if we set $x=(x_1,x_2)$, $y=(y_1,y_2)$ and
  assume $0<y_1<x_1$, then $\phi(x,y)=x_1-y_1+O((x_2-y_2)^2)$.
\end{lemma}
\begin{proof}
  See p. 144 in \cite{Sogge:93book}.
\end{proof}

With Lemma \ref{lemma:distance function} at hand,
we are ready to prove Proposition \ref{prop:example}.

\begin{proof}
  [Proof of Proposition \ref{prop:example}]
We work in the above coordinates and let
$$
\Omega_\mu =\bigl\{x: \delta/2C_0\le x_1\le 2C_0\delta,\;
|x_2|\le\eps_1 \mu^{-1/2}\bigr\},\;0<\eps_1\ll \delta,
$$
where $C_0>0$ is chosen as on p. 144 in \cite{Sogge:93book}. The region $\Omega_\lambda$  is defined similarly.
Take $\alpha\in C^\infty_0(-1,1)$ and set
\begin{equation}
\label{eq:f}
f(y)=\alpha(y_1/\eps_1)\,\alpha(\ld^\frac12 y_2/\eps_1)\,e^{i\ld y_1},
\end{equation}
\begin{equation}
\label{eq:g}
g(z)=\alpha(z_1/\eps_1)\,\alpha(\mu^\frac12 z_2/\eps_1)\,e^{i\mu z_1}.
\end{equation}
Denote by $\epsilon=\ld/\mu$.
Then similar to Chapter 5 in \cite{Sogge:93book}, we estimate
\[\int_{\Omega_\mu}|\mathcal T_\ld f(x)\mathcal T_\mu g(x)|^2dx.\]
Indeed, for $x\in\Omega_\mu$, we have

\begin{align*}
&|\mathcal T_\ld f(x)|^2
=\iint_{\Omega_\ld^2}e^{i\lambda (d_\g(x,y)-d_g(x,y')-[(x_1-y_1)-(x_1-y'_1)]}a(x,y,\lambda)\,\alpha(y_1/\eps_1)\,\alpha(\ld^\frac12 y_2/\eps_1)\\
&\qquad\qquad\times \overline{a(x,y',\lambda)\,\alpha(y'_1/\eps_1)\,\alpha(\ld^\frac12 y'_2/\eps_1)}\,dydy',
\end{align*}
Notice that by Lemma \ref{lemma:distance function}, the phase function equals $O(|x_2-y_2|^2)+O(|x_2-y'_2|^2)$. Since $|x_2|\le \eps_1\mu^{-1/2}$
and $|y_2|, |y'_2|\le \eps_1\ld^{-1/2}$, we see that
the phase in the exponent is of order $\eps_1^2$ on $\Omega_\mu$,
and the oscillation is eliminated in the integrand
by choosing $\eps_1$ small. Thus on $\Omega_\mu$
\[|\mathcal T_\ld f(x)|^2\gtrsim |\Omega_\ld|^2=\ld^{-1}.\]
Similarly, \[|\mathcal T_\ld g(x)|^2\gtrsim |\Omega_\mu|^2=\mu^{-1},\]
Thus, $\|\mathcal T_\ld f\mathcal T_\mu g\|_{L^2(\Omega_\mu)}$ is bounded below by $\mu^{-3/4}\ld^{-1/2}.$
On the other hand,
$\|f\|_2\|g\|_2\le c\, (\ld\mu)^{-1/4}$
for $f$ and $g$ given by \eqref{eq:f}\,\eqref{eq:g},
we have
\[(\ld\mu)^\frac12\|\mathcal T_\ld f\,\mathcal T_\mu g\|_2\,\Big/\,\bigl(\|f\|_2\,\|g\|_2\bigr) \ge C_{\eps_1}\ld^{1/4}.
\]\end{proof}
This example
exhibits the concentration of eigenfunctions
along a tubular neighborhood of
a geodesic leading to the sharpness of
the bilinear spectral projector estimate \eqref{eq:BGT estimate-1},
where our bilinear generalization of
the main result in \cite{Sogge:Tohoku} is motivated.
\begin{remark}
  Comparing this example with \eqref{eq:main result3},
  one may suspect that \eqref{eq:main result3}
 can be further refined. Indeed, one may observe that the example
  suggests the possibility of refining \eqref{eq:main result3}
  by strengthening the $L^2$-norm of the eigenfunction
  corresponding to the higher frequency on the right side to
  a $\lambda^{-\frac12}$-neighborhood of the\emph { same}
  geodesic segment for the lower frequency eigenfunction.
  An interesting problem would be to see if the following
  refinement of \eqref{eq:main result3} is valid:
  \begin{multline}
    \label{eq:conjecture}
    \|e_\ld e_\mu\|_{L^2(M)}
    \\
    \leq C_{\eps_0}\,\ld^\frac{1}{4}
   \sup_{\gamma\in\it\Pi}\, \biggl[\biggl(\int_{T_{\ld^{-\frac12+\eps_0}}
   (\gamma)}|e_\mu(x)|^2\,dx\biggr)\,
    \biggl(\int_{T_{\ld^{-\frac12+\eps_0}}(\gamma)}|e_\ld(x)|^2\,dx\biggr)\biggr]^\frac14.
  \end{multline}
\end{remark}
%
%
\section{Microlocal Kakeya-Nikodym averages}
\subsection{Basic notions}
In view of $\chi_\lambda e_\lambda=e_\lambda$
and $\chi_\mu e_\mu=e_\mu$, we are reduced  to estimating
$\|\mathcal T_\lambda f\mathcal T_\mu g\|_{L^2}$.
By scaling, we may assume the injectivity radius of $M$ is large enough, say ${\rm inj}\,M>10$.
We use partitions of unity on $M$ to reduce the $L^2$ integration of $\mathcal T_\lambda f\mathcal T_\mu g$
on the geodesic ball $B(x_0,\delta)$ with $\delta >0$ small.
In view of the property of supp $a$, we may apply partition of unity once more and assume ${\rm supp}\, f\subset B(y_0,\delta)$
and supp $g\subset B(z_0,\delta)$
for some $y_0$ and $z_0$ satisfying
\[1\le d_g(x_0,y_0),\, d_g(x_0,z_0)\le 2.\]

\medskip

Next, we need to choose a suitable coordinate system to simplify the calculations on a larger ball $B(x_0,10)$.
As in \cite{Sogge:Tohoku} and \cite{BSRef}, we shall use Fermi coordinate system about the geodesic $\gamma$ connecting $x_0$ and $y_0$.
Let $\gamma^\bot$ be the geodesic through $x_0$ perpendicular to $\gamma$.
The Fermi coordinates about $\gamma$
is defined on the ball $B(x_0,10)$, where the image of
$\gamma^\bot\cap B(x_0,10)$ in the
resulting coordinate system is parameterized by
$s\mapsto \{(s,0)\}$. All the horizontal segments are parameterized by $s\mapsto {(s,t_0)}$ and we have
$$d_\g((s_1,t_0),(s_2,t_0))=|s_1-s_2|.$$
Clearly, in our coordinate system, $y_0$ is on the 2nd coordinate axis, and  $z_0$ is a point satisfying $1\le d_g(z_0,(0,0))\le 2$.\medskip

Therefore, if we set $y=(s,t),\, z=(s',t')$ in this coordinate system, we may write $\mathcal T_\lambda f$ and $\mathcal T_\mu g$ locally
as
\begin{align}
\mathcal T_\lambda f(x)=&\int_{\mathbb R^2}e^{i\lambda d_\g(x,(s,t))}
   a(x,(s,t),\lambda)\,f(s,t)\,dsdt,\\
\mathcal T_\mu g(x)=&\int_{\mathbb R^2}e^{i\mu d_\g(x,(s',t'))}
   a(x,(s',t'),\mu)\,g(s',t')\,ds'dt'.
\end{align}
Moreover, by noting that $1\le d_\g(x_0,y_0), d_\g(x_0,z_0)\le 2$ and
$y\in B(y_0, \delta)$, $z\in B(z_0, \delta)$, we shall assume
$$\max\bigl\{|s|, |t- d_\g(y_0,x_0)|, |d_g((s',t'),z_0)|\bigr\}\le\delta.$$
We remark that we are at liberty to take $\delta$ to be small when necessary.

%
%

\subsection{Preliminary reductions}

First of all, we deal with the case
when the angle between $\gamma$
and the geodesic $\gamma'$ connecting $x_0$ and $z_0$ is bounded below by some $\eps_2>0$.
To do this, we shall use the geodesic normal coordinates around $x_0$. Set $\{e_1, e_2\}$ to be the orthonormal basis in $T_{x_0} M$,
where the metric $\g$ at $x_0$ is normalized,
such that $e_1$ is the tangent vector of $\gamma^\bot$
at $x_0$ and $-e_2$ is the tangent vector of $\gamma$ at $x_0$
if $\gamma$ is oriented from $x_0$ to $y_0$.
Let $\{\omega_1,\omega_2\}$ be
the dual basis of $\{e_1,e_2\}$
and set $\{x_j=\omega_j\circ \exp_{0}^{-1}\}_{j=1,2}$
to be the Riemannian geodesic normal coordinate system on $B(x_0,10)$,
where $x_0=0$ and
$\gamma$ is parameterized by $x_2\mapsto \{(0,x_2)\}$,
whereas
$\gamma^\bot$ is parameterized by $x_1\mapsto \{(x_1,0)\}$ with $|x_1|\leq5 $.
Let $\theta_0=\theta(z_0)$ be
such that $z_0=d_\g(x_0,z_0)(\cos\theta_0,\sin\theta_0)$,
where the angular variable is oriented in clockwise direction.
It follows that $\gamma'^\bot$ is given by
$r\mapsto \exp_{0}\bigl( (r\cos\varphi_0,r\sin\varphi_0) \bigr)$
with $\varphi_0=\theta_0+\frac\pi2$ and $|r|<5$.

\medskip

Writing
\begin{equation}
  \label{eq:y,z in GNC}
  y=(r_1\cos \theta_1,r_1\sin\theta_1),\;\;z=(r_2\cos\theta_2,r_2\sin\theta_2)
\end{equation}
in geodesic normal coordinates,
we have
\begin{align}
\mathcal T_\lambda f(x)=&\iint e^{i\lambda d_\g(x,(r_1,\theta_1))}
   a(x,(r_1,\theta_1),\lambda)\,f(r_1,\theta_1)\,dr_1d\theta_1,\\
\mathcal T_\mu g(x)=&\iint e^{i\mu d_\g(x,(r_2,\theta_2))}
   a(x,(r_2,\theta_2),\mu)\,g(r_2,\theta_2)\,dr_2d\theta_2.
\end{align}
We recall the following fact.
\begin{proposition}
  \label{prop:reduction}
  Let $\eps_2>0$ be a small parameter. Assume
  $\bigl|\,\theta(z_0)+ \frac\pi 2\,\bigr|\ge \eps_2$
  and $ \bigl|\,\theta(z_0)- \frac\pi 2\,\bigr|\ge \eps_2$.
  If we choose $\delta$ small enough depending on $\eps_2$,
  there exists $C$ such that
  \begin{equation}
  \bigl\|\mathcal T_\ld f\,\mathcal T_\mu g\bigr\|_2
  \le C(\ld\mu)^{-1/2}\|f\|_2\|g\|_2.
\end{equation}
\end{proposition}
Thus in order to prove Theorem \ref{theorem:main},
it suffices to consider either
$\bigl|\,\theta(z_0)+ \frac\pi 2\,\bigr|\le \eps_2$
or $\bigl|\,\theta(z_0)- \frac\pi 2\,\bigr|\le \eps_2$.
This confines $z_0$ in a small neighborhood of
the geodesic $\gamma$ by compressing $\gamma$ and $\gamma'$
to be almost parallel with each other.
\medskip

Essentially, this proposition is proved in
\cite{BurqGerardTzvetkov:Invention} based on
the following lemma.

\begin{lemma}\label{lem:BGT lemma}
Let $y=\exp_0(r(\cos\theta,\sin\theta))$ and $\phi_r(x,\theta)=d_\g(x,y)$.
For every $0<\eps_2<1$, there exists $c>0$, $\delta_1>0$ such that for every $|x|<\delta_1$,
\beq\label{eq:BGT lemma}
\bigl |{\rm det}\,\bigl(\nabla_x\partial_\theta\phi_r(x,\theta),
\nabla_x\partial_{\theta'}\phi_{r'}(x,\theta')\bigr)\bigr|\ge c,
\eeq
if $|\theta-\theta'|\ge\eps_2$ and $|\theta+\pi-\theta'|\ge \eps_2$.
In addition, for every $\theta\in[0,2\pi]$,
\begin{equation}
  \bigl|{\rm det}\bigl[\nabla_x\partial_\theta\phi_r(x,\theta),
  \nabla_x\partial^2_\theta\phi_r(x,\theta)\bigr]\bigr|\ge \,c.
\end{equation}
\end{lemma}

This is an immediate consequence of the following fact.
\begin{lemma}
  \label{lemma:cosphere}
  Let $y\mapsto\kappa(y)=\exp_0^{-1}(y)$ be the geodesic normal coordinates vanishing at $x_0$, as described above.
  Then we have
  \begin{equation}
    \label{eq:cosphere relation}
    \nabla_x d_\g(x,y)\Bigl|_{x=x_0}=\kappa(y)/|\kappa(y)|.
  \end{equation}
\end{lemma}
\begin{proof}
  Relation \eqref{eq:cosphere relation} is equivalent to Gauss' lemma. See \cite{Sogge:Tohoku} and \cite{BurqGerardTzvetkov:Invention}.
\end{proof}
\begin{remark}
  We see from this lemma that the set of points
  $\{\nabla_x d_\g(x,y):x=x_0, d_\g(x_0,y)\in (1/2,2)\}$
  is exactly the cosphere at $x_0$, {\em i.e.}
  $$S^*_{x_0}M=\Bigl\{\xi:\sum g^{jk}(x_0)\xi_j\xi_k=1\Bigr\},\quad \g^{ij}=(\g_{ij})^{-1}.$$
  The map $y\mapsto \kappa(y)$ is a local radial isometry. See \cite{Sogge:Tohoku}.
\end{remark}

We  sketch the proof of Proposition \ref{prop:reduction} briefly
for completeness.
In our situation, we have $\theta(y_0)=-\frac\pi2$.
Fixing a parameter $\eps_2>0$, we assume
$\bigl|\theta(z_0)+ \frac\pi 2\bigr|\ge \eps_2$
and $ \bigl|\theta(z_0)- \frac\pi 2\bigr|\ge \eps_2$.
Since $y\in B(y_0,\delta), z\in B(z_0,\delta)$
given by \eqref{eq:y,z in GNC},
we may choose $\delta<\delta_1$.
As a consequence, we have
$|\theta_1-\theta_2|\ge c\eps_2$
and $|\theta_1+\pi-\theta_2|\ge c\eps_2$ with some $c>0$.
By Schur's test, it suffices to show
\begin{equation}\label{eq:kernel good}
  \bigl|K(\theta_1,\theta_2,\theta_1',\theta_2')\bigr|
  \leq C (\mu|\theta_2-\theta_2'|+\lambda|\theta_1-\theta_1'|)^{-10},
\end{equation}
where
\[K(\theta_1,\theta_2,\theta_1',\theta_2')
=\int e^{i\,\Psi_{\lambda,\,\mu}(x;\,\theta_1,\theta_2,\theta_1',\theta_2')}
A(x;\,\theta_1,\theta_1',\theta_2,\theta_2')\,dx,
\]
\[A(x;\,\theta_1,\theta_1',\theta_2,\theta_2')
=a(x,(r_1,\theta_1),\lambda)
\overline{a(x,(r_1,\theta_1'),\lambda)}
a(x,(r_2,\theta_2),\mu)
\overline{a(x,(r_2,\theta_2'),\mu)},\]
\[\Psi_{\lambda,\,\mu}(x;\,\theta_1,\theta_2,\theta_1',\theta_2')
=\lambda\bigl(\phi_{r_1}(x,\theta_1)-\phi_{r_1}(x,\theta_1')\bigr)
+\mu\bigl(\phi_{r_2}(x,\theta_2)-\phi_{r_2}(x,\theta_2')\bigr).\]
For all multi-index $\alpha, |\alpha|\leq 10$, Lemma \ref{lem:BGT lemma} and the above formula give
\[|\nabla_x\Psi_{\lambda,\mu}|\ge C(\lambda|\theta_1-\theta_1'|
+\mu|\theta_2-\theta_2'|),\;|\partial^\alpha_x\Psi_{\lambda,\mu}|\le C(\lambda|\theta_1-\theta_1'|
+\mu|\theta_2-\theta_2'|).\]
Now \eqref{eq:kernel good} follows from integration by parts.

%
%
\subsection{Decomposition of the phase space and microlocal Kakeya-Nikodym averages}
We will employ the strategy introduced by \cite{Blair Sogge:refine},
where a  microlocal refinement of
Kakeya-Nikodym averages are exploited.
From now on, we shall always assume
\[\Bigl|\theta(z_0)+\frac\pi2\Bigr|\leq \eps_2\ll 1\]
where
$\frac\pi2=-\theta(y_0)$.
Recall that we may write, modulo trivial errors,
\begin{align}
  \chi_\ld f(x)\approx& \ld^\frac12\int_{\R^2}e^{i\ld d_\g(x,y)}a_\ld(x,y)\, f(y)\,dy,\\
  \chi_\mu g(x)\approx& \mu^\frac12\int_{\R^2}e^{i\mu d_\g(x,z)}a_\mu(x,z)\, g(z)\,dz,
\end{align}
with supp $f\subset B(y_0,\delta)$,
supp $g\subset B(z_0,\delta)$ and $x\in B(0,\delta)$.\medskip

As discussed in the last section, we may choose $\eps_2>0$ sufficiently small
to make $z_0$ to be within an fixed small neighbourhood of $\gamma$.

\medskip

To decompose the phase space, we shall use the geodesic flow
$\Phi_\tau(y,\xi)$  on the cosphere bundle $S^*M$, which
starts from $y$ in direction of $\xi\in S^*_y M$.
We use the Fermi coordinates around $\gamma$ to write
\[\bigl(y(\tau),\xi(\tau)\bigr)=\Phi_\tau(y,\xi),\;\;(y(0),\xi(0))=(y,\xi),\]
where $\xi(\tau)$ is the unit cotangent vector in $T^*_{y(\tau)}M$.
Define
$\Theta:(y,\xi)\in S^*M\rightarrow \R\times\R$
by
\[\Theta(y,\xi)=\left(\Pi_{y_1}\Phi_{\tau_0}(y,\xi),\;\;
\frac{\Pi_{\xi_1}\Phi_{\tau_0}(y,\xi)}{|\Pi_\xi\Phi_{\tau_0}(y,\xi)|}\right),\]
where $\tau_0$ is chosen so that
$y_2(\tau_0)=\Pi_{y_2}\Phi_{\tau_0}(y,\xi)=0$.
By $\Pi_\diamondsuit $, we mean
the projection to the component of $\diamondsuit$-variable.

\begin{remark}
\label{rmk}
  As in \cite{Blair Sogge:refine},
  we require $|\xi_1|<\delta$ with $\delta$ small enough
  with $y\in B(y_0,\,C_0\delta)$.
  Moreover,  $\Theta $ is constant on the orbit of $\Phi$
  and $|\Theta(y,\xi)-\Theta(z,\eta)|$ can be used
  as a natural distance function between geodesics
  passing respectively through $(y,\xi)$ and $(z,\eta)$.
\end{remark}

Next, we microlocalize $\chi_\ld f$ and $\chi_\mu g$  by introducing  smooth functions
$\alpha_{1}(y)$ and $\alpha_{2}(z)$ adapted respectively
to the ball $B(y_0, 2\delta)$ and $B(z_{0},2\delta)$ and setting
\begin{align}
  Q^{\nu}_\theta(y,\xi)=&\alpha_1(y)\,
  \beta(\theta^{-1}\Theta(y,\xi)+\nu)\,\Upsilon(|\xi|/\ld)\label{eq:symbol q}\\
  P^{\upsilon}_\theta(z,\eta)=&\alpha_2(z)\,\beta(\theta^{-1}\Theta(z,\eta)+\upsilon)\,
  \Upsilon(|\eta|/\mu)\label{eq:symbol p}
\end{align}
where
$\lambda^{-1/2}\leq \theta\leq 1$, $\nu,\upsilon\in \Z^2$, with  $\beta$ smooth
such that
\begin{equation}
  \label{eq:sum-beta}
  \sum_{\nu\in \Z^2}\beta({}\cdot{}+\nu)=1,\quad\text{supp}\,\beta\subset\{x\in \R^2:|x|\leq 2\},
\end{equation}
and $\Upsilon\in C^\infty_0(\R)$ is
supported in $[c,c^{-1}]$ for some $c>0$.\\

Let us take a look at the symbols
$Q^{\nu}_{\theta}(y,\xi)$ and $P^{\upsilon}_\theta(z,\eta)$.
First, we define $ \beta(\theta^{-1}\Theta(y,\xi)+\nu)$ and
$\beta(\theta^{-1}\Theta(z,\eta)+\upsilon)$
on the cosphere bundle.
Since these two functions are of degree zero in the cotangent variables,
we then extend them homogeneously
to the cotangent bundle with the zero section removed.
The above $Q^{\nu}_{\theta}(y,\xi)$
and $P^{\upsilon}_\theta(z,\eta)$ are well-defined for $\xi\neq 0,\,\eta\neq0$.
Given $\xi$, $\beta(\theta^{-1}\Theta(y,\xi)+\nu)=0$
unless $y $ belongs to a tubular neighborhood
of $\gamma_{\nu}$, where
\[\gamma_\nu=\bigl\{y(\tau):-2\leq \tau\leq 2,\,
  (y(\tau),\xi(\tau))=\Phi_\tau(y,\,\xi),
  \Theta(y,\,\xi)+\theta\nu=0\bigr\}.\]
Moreover, if we set $\nu=(\nu_1,\nu_2)$,
the direction of $\gamma_{\nu}$ at $y(\tau_0)$
is determined by $\theta \nu_{2}$
and is independent of $\ld$.
Since $(y,\xi)=\Phi^{-1}_{\tau_0}(y(\tau_0),\xi(\tau_0))$
and $y(\tau_0)=(y_1(\tau_0),0)$
with $y_1(\tau_0)=\theta\nu_{1}+O(\theta)$,
one easily finds that
$y\in T_{C_1\theta}(\gamma_{\nu})$,
for some $C_1\geq 1$.
Similar statements hold for $P^{\upsilon}_\theta(z,\eta)$.

\medskip

Let $Q^{\nu}_{\theta}(x,D)$, $P^{\upsilon}_{\theta}(x,D)$
 be the pseudo-differential operators
 associated to the symbols defined in \eqref{eq:symbol q} and
  \eqref{eq:symbol p} respectively.
We next record some properties of
$Q^{\nu}_{\theta}(y,D)$ and $P^{\upsilon}_\theta(z,D)$.
The first lemma  indicates that these two kinds of operators provide a natural
microlocal wave-packet decomposition in the phase space
for 2-dimensional manifolds.
\begin{lemma}
  \label{lemma:first microlocalization}
  If $\ld^{-1/2+\eps}\leq \theta\leq 1$ with $\eps>0$ fixed,
  the symbols $Q^{\nu}_\theta$ and $P^{\upsilon}_\theta$ belong to a bounded subset of $S^0_{1/2+\eps,1/2-\eps}$.
%
    Then there is $C_{\eps}$ and $C_2\geq C_1$
  such that for $\lambda^{-1/2+\eps}\leq\theta \leq 1$,
  we have
  \begin{align}
    \label{eq:recurrence1}
   & \|Q^\nu_\theta(x,D)\,f\|_{L^2}
   \leq C_{\eps}\|f\|_{L^2(T_{C_2\theta}(\gamma_\nu))}+C_N\ld^{-N}\|f\|_2\\
    \label{eq:recurrence2}
   & \|P^\upsilon_\theta(x,D)\,g\|_{L^2}
   \leq C_{\eps}\|g\|_{L^2(T_{C_2\theta}(\gamma_\nu))}+C_N\mu^{-N}\|g\|_2.
  \end{align}
    Moreover, for any integer $N\geq 0$, one may write
  \begin{align}
    \label{eq:wavepacket1}\chi_\ld \,f
    =&\sum_{\nu\in\Z^2}\chi_\ld\circ Q^{\nu}_{\theta}(x,D)\,f+R_\ld \,f,\ \ \ {\rm \it if}\ {\rm supp}\,f\subset B(y_0,\delta),\\
    \label{eq:wavepacket2}\chi_\mu\, g
    =&\sum_{\upsilon\in\Z^2}\chi_\mu\circ P^{\upsilon}_{\theta}(x,D)\,g+R_\mu \,g,\ \ \ {\rm \it if}\ {\rm supp}\,g\subset B(z_0,\delta),  \end{align}
  with $\|R_\ld\|_{L^2\rightarrow L^\infty}\lesssim \ld^{-N},
  \|R_\mu\|_{L^2\rightarrow L^\infty}\lesssim \mu^{-N}$.
\end{lemma}
\begin{proof}
  That $Q^{\nu}_{\theta}(y,\xi)\in S^0_{1/2+\eps,1/2-\eps}$
  has already been proved in \cite{Blair Sogge:refine}.
   If we use $\mu\geq\ld$, we get $\mu^{-1}\ld^{1/2-\eps}\leq \mu^{-1/2-\eps}$,
   and the same calculation as for $Q^{\nu}_{\theta}$
   yields that the $P^{\upsilon}_{\theta}(z,D)$
   belong to a bounded subset of pseudodifferential operators of order zero and type $(1/2+\eps,1/2-\eps)$.
  To see \eqref{eq:recurrence1},
  one observes that the kernel $K^\nu_{\theta}(x,y)$
  of the operator $Q^{\nu}_{\theta}$ is bounded by
  $O(\ld^{-N})$ if $y$ does not belong to
  $T_{C_2\theta}(\gamma_\nu)$ for some large $C_2>C_1$
  by using integration by parts.
  We can deduce  \eqref{eq:wavepacket1} from  \eqref{eq:sum-beta}.
  In fact, if we recall the process of constructing parametrix for the half wave
  operator $e^{it\sqrt{-\Delta_{\g}}}$ in \cite{Sogge:93book},
  we may use integration by parts to see that  in \eqref{eq:wavepacket1},
  one may assume $\widehat f(\xi)=0$ if
 $|\xi|\not\in [c\lambda,\,C\lambda]$ up to some terms of the form $\mathcal R_{\lambda}f$.
 It suffices to see the difference of $f(x)$ and $\sum_{\nu}Q^{\nu}_{\theta}(x,D)f(x)$
 is of the form $\mathcal R_{\lambda}f(x)$.
 This is easy due to the fact that $\Upsilon(|\xi|/\ld)=1$ on the support of $\widehat f$
 by choosing suitable $c,\,C$ and $(1-\alpha_{1}(x))f(x)=0$. Now \eqref{eq:sum-beta} yields
 $$\alpha_{1}(x)f(x)=\sum_{\nu}Q^{\nu}_{\theta}(x,D)f(x).$$
  Similar argument yields \eqref{eq:recurrence2}
  and \eqref{eq:wavepacket2} .
\end{proof}

Now, we recall the microlocal Kakeya-Nikodym norm in  \cite{Blair Sogge:refine},
corresponding to frequency $\ld$ and  $\theta_0=\ld^{-1/2+\eps_0}$
\begin{equation}
  \label{def:MKN}
  \vertiii{f}_{MKN(\ld,\eps_0)}=\sup_{\theta_0\leq \theta\leq 1}
  \Bigl(\sup_{\nu\in\Z^2}\theta^{-1/2}\|Q^\nu_{\theta}(x,D)f\|_{L^2(\R^2)}\Bigr)+\|f\|_{L^2(\R^2)}.
\end{equation}
As  pointed out in \cite{Blair Sogge:refine},
the maximal microlocal concentration of $f$ about all
unit geodesics in the scale of $\theta$ amounts to the quantity
\[\sup_{\nu\in\Z^2}\theta^{-1/2}\|Q^\nu_{\theta}(x,D)f\|_{L^2(\R^2)}.\]
From Lemma \ref{lemma:first microlocalization}, one can prove
 $\vertiii{f}_{MKN(\ld,\eps_0)}\leq C_{\eps_0}\vertiii{f}_{KN(\ld,\eps_0)}$.
 We refer to \cite{Blair Sogge:refine} for more details. Similarly, for the same $\theta_0$, we can define
 \begin{equation}
  \label{def:MKN'}
  \vertiii{g}_{MKN'(\ld,\eps_0)}=\sup_{\theta_0\leq \theta\leq 1}
  \Bigl(\sup_{\nu\in\Z^2}\theta^{-1/2}\|P^\nu_{\theta}(x,D)g\|_{L^2(\R^2)}\Bigr)+\|g\|_{L^2(\R^2)},
\end{equation}
again by Lemma \ref{lemma:first microlocalization}, we see that $\vertiii{g}_{MKN'(\ld,\eps_0)}\leq C_{\eps_0}\vertiii{g}_{KN(\ld,\eps_0)}$.

\medskip

We will use the following fact in the next section.
\begin{lemma}
  \label{lemma:symbol class of P Q}
  For any $\eps>0$,
   there exists some $C_{\eps}>0$ such that for all $\ld^{-1/2+\eps}\leq \theta\leq 1$, 
  \begin{equation}
    \label{eq:L2 bdd of composed operators}
    \bigl\|\sum_{\nu}(Q^\nu_\theta)^*\circ Q^\nu_\theta\, f\bigr\|_{L^2}\leq C_{\eps}\|f\|_{L^2},\;
    \bigl\|\sum_{\upsilon}(P^\upsilon_\theta)^*\circ P^\upsilon_\theta\, g\bigr\|_{L^2}\leq C_{\eps}\|g\|_{L^2}.
  \end{equation}
\end{lemma}
\begin{proof}
   The $L^2$-estimates 
   \eqref{eq:L2 bdd of composed operators} are valid thanks to \eqref{eq:sum-beta}
   and the  classical  calculus of
   pseudo-differential operators of type $(1/2+\varepsilon, 1/2-\varepsilon)$ with $\varepsilon>0$.
\end{proof}

We describe next the kernels of the operators  $\chi_\ld Q^{\nu}_{\theta}:=(\chi_\ld\circ Q^{\nu}_{\theta})(x,D)$ and $\chi_\mu P^{\upsilon}_{\theta}:=(\chi_\mu\circ P^{\upsilon}_{\theta})(x,D)$ following \cite{Blair Sogge:refine}.
\begin{lemma}
  \label{lemma:Kernels}
  Denote by $(\chi_\ld Q^\nu_\theta)(x,y)$ and $(\chi_\mu P^\upsilon_\theta)(x,z)$ the kernels of the pseudodifferential operators
  $\chi_\ld Q^\nu_\theta(x,D)$ and $\chi_\mu P^\upsilon_\theta(x,D)$ respectively. Assume $\theta\in[C_0\theta_0,1]$ with $\theta_0=\ld^{-1/2+\eps}$ and $C_0\gg 1$.
  We can find a uniform constant $C$ so that for each $N=1,2,3,\ldots$, we have
  \begin{equation}
    \label{eq:kernels decay Q}
    |(\chi_\ld Q^\nu_\theta)(x,y)|\leq C_N \ld^{-N},\;\text{if}\;x\not\in T_{C\theta}(\gamma_\nu) \;\text{or}\;y\not\in T_{C\theta}(\gamma_\nu),
  \end{equation}
  and
  \begin{equation}
    \label{eq:kernels decay P}
    |(\chi_\mu P^\upsilon_\theta)(x,z)|\leq C_N \mu^{-N},\;\text{if}\;x\not\in T_{C\theta}(\gamma_\upsilon) \;\text{or}\;z\not\in T_{C\theta}(\gamma_\upsilon).
  \end{equation}
  Furthermore,
  \begin{equation}
    \label{eq:osc-Q}
    (\chi_\ld Q^\nu_\theta)(x,y)=\ld^\frac12e^{i\ld d_g(x,y)}a_{\nu,\theta}(x,y)+O_N(\ld^{-N}),
  \end{equation}
  \begin{equation}
    \label{eq:osc-P}
    (\chi_\mu P^\upsilon_\theta)(x,z)=\mu^\frac12e^{i\mu d_g(x,z)}b_{\upsilon,\theta}(x,z)+O_N(\mu^{-N}),
  \end{equation}
  where we have the uniform bounds
  \begin{equation}
    \label{eq:a b estimate}
    |({\nabla^{\bot}_x})^\alpha a_{\nu,\theta}(x,y)|\leq C_\alpha \theta^{-|\alpha|},\;
    |({\nabla^{\bot}_x})^\alpha b_{\upsilon,\theta}(x,z)|\leq C_\alpha \theta^{-|\alpha|},
  \end{equation}
  and
  \begin{equation}
    \label{eq:a estimate along curve}
    |\partial^j_t a_{\nu,\theta}(x,x_\nu(t))|\leq C_j, x\in \gamma_\nu=\{x_\nu(t)\},
  \end{equation}
   \begin{equation}
    \label{eq:b estimate along curve}
   |\partial^\ell_t b_{\upsilon,\theta}(x,x_\upsilon(t))|\leq C_\ell, x\in \gamma_\upsilon=\{x_\upsilon(t)\},
  \end{equation}
  where $\nabla^{\bot}_{x}$ denotes the directional derivative along the direction perpendicular to
  the geodesics $\{x_{\nu}(t)\}$ with $\nu=\nu$ or $\upsilon$ and
  \[\gamma_\nu=\bigl\{z_\nu(\tau):-2\leq \tau\leq 2,(z_\nu(\tau),\eta_\nu(\tau))=\Phi_\tau(z_\nu,\eta_\nu),
  \theta^{-1}\Theta(z_\nu,\eta_\nu)+\nu=0\bigr\}.\]
\end{lemma}

\begin{proof}
  The properties for $(\chi_\ld Q^\nu_\theta)(x,y)$ are exactly the same as in
  \cite{Blair Sogge:refine}, and the proof is identical to that of Lemma 3.2 in
  \cite{Blair Sogge:refine}. Since $\theta\ge \mu^{-\frac1 2+\eps}$, the properties for $(\chi_\ld P^\upsilon_\theta)(x,z)$ follows from the same proof.
\end{proof}

On account of the above lemma, we have the following fact which will be used in the 
next section.
\begin{lemma}
  \label{lemma:orthogonality}

Assume $\theta\ge\theta_0$ and $N_1$ is fixed. Then there exists $C_0\gg 1$, when $|\nu-\tilde\nu|+|\upsilon-\tilde\upsilon|\ge C_0$ and $|\nu-\upsilon|,|\tilde\nu-\tilde\upsilon|\le N_1$, we have
\[\left|\int \chi_\lambda Q_\theta^\nu h_1(x)\,\chi_\mu P_\theta^{\upsilon} h_2(x)\,\overline{\chi_\lambda Q_\theta^{\tilde\nu} h_3(x)}\,\overline{\chi_\mu P_\theta^{\tilde\upsilon} h_4(x)}\,dx\right|\le C_N\mu^{-N}\prod_{j=1}^4\|h_j\|_2.\]

\end{lemma}
\begin{proof}
To get $O_N(\mu^{-N})$ decay as claimed, we need to split into two cases depending on the size of $\mu$. Assume first $\mu\ge \lambda^2$.

It suffices to consider the kernel 
\[K(y,z,\tilde y,\tilde z)=\int \chi_\lambda Q_\theta^\nu (x,y)\,\chi_\mu P_\theta^{\upsilon} (x,z)\,\overline{\chi_\lambda Q_\theta^{\tilde\nu} (x,\tilde y)}\,\overline{\chi_\mu P_\theta^{\tilde\upsilon} (x,\tilde z)}dx.\]

Indeed, by Lemma \ref{lemma:Kernels}, up to a $O_N(\mu^{-N})$ error, we can restrict the domain of integration here to $\Omega=T_{C\theta}(\gamma_{\upsilon})\cap T_{C\theta}(\gamma_{\tilde\upsilon}).$ 

Plugging \eqref{eq:osc-P} into the expression of $K(y,z,\tilde y,\tilde z)$, we get

\[K(y,z,\tilde y,\tilde z)=\mu\int_\Omega b(x,y, z, \tilde y, \tilde z)e^{i\mu (d_g(x,z)-d_g(x,\tilde z))}dx+O_N(\mu^{-N}),\]

where 
\[b(x,y, z, \tilde y, \tilde z)=\chi_\lambda Q_\theta^\nu (x,y)\,\overline{\chi_\lambda Q_\theta^{\tilde\nu} (x,\tilde y)}\,b_{\upsilon,\theta}(x,z)\,\overline {b_{\tilde\upsilon,\theta}(x,\tilde z)}.\]
It is easy to see that $b(x,y, z, \tilde y, \tilde z)$ satisfies
\[|\nabla^\alpha_x b(x,y, z, \tilde y, \tilde z)|\le C\lambda^{|\alpha|+1}.\]

Now we consider the phase function
\[\mu(d_g(x,z)-d_g(x,\tilde z)).\]
The gradient reads
\[\mu\nabla_x(d_g(x,z)-d_g(x,\tilde z)).\]
We claim that for $C_0$ big enough, there exists some $c_0>0$, such that
\[|\nabla_x(d_g(x,z)-d_g(x,\tilde z))|\ge c_0\theta,\]
then our lemma follows from simple integration by parts argument.

Indeed, since $x\in T_{C\theta}(\gamma_{\upsilon})\cap T_{C\theta}(\gamma_{\tilde\upsilon})$, $z\in T_{C\theta}(\gamma_{\upsilon})$ and $\tilde z\in T_{C\theta}(\gamma_{\tilde\upsilon})$, we see that \[|\nabla_x(d_g(x,z)-d_g(x,\tilde z))|\gtrsim |\upsilon-\tilde\upsilon|\theta,\]
noticing that \[|\upsilon-\tilde\upsilon|\ge|\nu-\tilde\nu|-|\nu-\upsilon|-|\tilde \nu-\tilde\upsilon|\ge|\nu-\tilde\nu|-2N_1,\]
thus for $C_0$ big enough,
\[|\upsilon-\tilde\upsilon|\ge\frac{1}{2}(|\upsilon-\tilde\upsilon|+|\nu-\tilde\nu|)-N_1\ge \frac{1}{2}C_0-N_1\ge c_0,\]
finishes the proof for the case $\mu\ge\lambda^2$.

Now we assume $\mu\le\lambda^2$, then again by Lemma \ref{lemma:Kernels}, up to a $O_N(\mu^{-N})=O_{2N}(\lambda^{-2N})$ error, we can further restrict the domain of integration in this case to $\Omega'=T_{C\theta}(\gamma_{\upsilon})\cap T_{C\theta}(\gamma_{\tilde\upsilon})\cap T_{C\theta}(\gamma_{\nu})\cap T_{C\theta}(\gamma_{\tilde\nu}).$

Similarly as above, by plugging \eqref{eq:osc-Q} and \eqref{eq:osc-P} into the expression of $K(y,z,\tilde y,\tilde z)$, we see that the resulting phase function is given by
\[\lambda(d_g(x,y)-d_g(x, \tilde y))+\mu(d_g(x,z)-d_g(x,\tilde z)).\]
The gradient reads
\[\lambda\nabla_x(d_g(x,y)-d_g(x,\tilde y))+\mu\nabla_x(d_g(x,z)-d_g(x,\tilde z)).\]
Let us denote $\nabla_x(d_g(x,y))=Y$, here $Y$ is a unit vector in $T_x M$, similarly denote $\nabla_x(d_g(x,\tilde y))=\widetilde Y$, $\nabla_x(d_g(x,z))=Z$ and $\nabla_x(d_g(x,\tilde z))=\widetilde Z$. By the separation conditions we have, it is easy to see that $\angle(Y,Z),\,\angle(\widetilde Y,\widetilde Z)\le N_1\theta$ and $\angle(Y,\widetilde Y)+\angle(Z,\widetilde Z)\ge C_0\theta$.

We Claim that
\begin{equation}
\label{angle separation}
\left|Y-\widetilde Y+\frac \mu \ld\big(Z-\widetilde Z\big)\right|=\left|\big(Y+\frac \mu\ld Z\big)-\big(\widetilde Y+\frac\mu\ld \widetilde Z\big)\right|\ge c\frac\mu\ld\theta,
\end{equation}
which implies the desired result using integration by parts. Indeed, it suffices to show that $\angle(Y+\frac\mu\ld Z,\widetilde Y+\frac\mu\ld \widetilde Z)$ is bounded below by some uniform constant times $\theta$. Note that $\angle(Y+\frac\mu\ld Z,Y),\,\angle(\widetilde Y,\widetilde Y+\frac\mu\ld \widetilde Z)\, ,\angle(Y+\frac\mu\ld Z,Z),\,\angle(\widetilde Z,\widetilde Y+\frac\mu\ld \widetilde Z)\le N_1\theta$, we have
\[\angle(Y+\frac\mu\ld Z,\widetilde Y+\frac\mu\ld \widetilde Z)\ge \angle(Y,\widetilde Y)-2N_1\theta,\]
similarly,
\[\angle(Y+\frac\mu\ld Z,\widetilde Y+\frac\mu\ld \widetilde Z)\ge \angle(Z,\widetilde Z)-2N_1\theta.\]
Thus for $C_0$ large enough,
\[\angle(Y+\frac\mu\ld Z,\widetilde Y+\frac\mu\ld \widetilde Z)\ge \frac 1 2(\angle(Y,\widetilde Y)+\angle(Z,\widetilde Z))-2N_1\theta\ge \frac 1 2 C_0\theta-2N_1\theta\ge c_0\theta,\]
finishes the proof.
\end{proof}

%
%
%
%
\section{Proof of the main theorem I: Orthogonality}
In this section, we use orthogonality argument to
reduce the proof of Theorem \ref{theorem:main} to
a specific bilinear estimate.
We use Lemma \ref{lemma:first microlocalization}
and Minkowski's inequality to estimate $\|\chi_\ld f\chi_\mu g\|_2$ by
\begin{align}
   \Bigl\|\sum_{|\nu-\upsilon|\leq M}&
   \chi_\ld Q^{\nu}_{\theta_0}f\;\;\chi_\mu P^{\upsilon}_{\theta_0} g\Bigr\|_2\label{eq:diagonal}\\
  +  \sum_{\ell=\log M/\log 2}^{O(\log \ld)}&
 \Bigl\|\sum_{2^\ell\leq|\nu-\upsilon|< 2^{\ell+1}}\chi_\ld Q^{\nu}_{\theta_0}f\;\;
 \chi_\mu P^{\upsilon}_{\theta_0} g\Bigr\|_2,
  \label{eq:off diagonal}
\end{align}
for certain dyadic $M$ large enough.
The square of \eqref{eq:diagonal} is estimated by
\begin{align}
 \label{eq:split}
 \biggl[\sum_{|\nu-\nu'|+|\upsilon'-\upsilon|\leq C_0}
   +\sum_{|\nu-\nu'|+|\upsilon'-\upsilon|\geq C_0}\biggr]
   \int \chi_\ld Q^{\nu}_{\theta_0}f(x)\;\chi_\mu P^{\upsilon}_{\theta_0} g(x)\;
  \overline{ \chi_\ld Q^{\nu'}_{\theta_0}f(x)}\;
  \overline{\chi_\mu P^{\upsilon'}_{\theta_0} g(x)}\;dx,
\end{align}
where  $|\nu-\upsilon|,|\nu'-\upsilon'|\leq M$.\medskip

By Lemma \ref{lemma:orthogonality}, the second term of \eqref{eq:split} is negligible by choosing $C_0$ 
sufficiently large.
\medskip

We can estimate the contribution of the first term as
\[\sum_{\upsilon\in\Z^2}\sum_{\nu:|\nu-\upsilon|\leq M}
\bigl\|\chi_\ld Q^{\nu}_{\theta_0}f\;\chi_\mu P^{\upsilon}_{\theta_0} g\bigr\|^2_2.\]
If we use the bilinear estimate \eqref{eq:BGT estimate-1},
we can estimate this sum by
\begin{align*}
  \ld^{\frac12}\sum_{\upsilon\in\Z^2}\bigl\|P^{\upsilon}_{\theta_0} g\bigr\|^2_2
  \sum_{\nu:|\nu-\upsilon|\leq M}\bigl\|Q^{\nu}_{\theta_0}f\bigr\|^2_2.
\end{align*}
By the $L^2$-orthogonality, we see the contribution of \eqref{eq:diagonal} is
\[\ld^{\frac{\eps_0}{2}}\|g\|_2\times\Bigl( \ld^{\frac12-\eps_0}\sup_\nu\|Q^\nu_{\theta_0}f\|^2_2\Bigr)^\frac12,\]
which corresponds to \eqref{eq:main result1}. Similarly, since the sum is symmetric, we can also bound \eqref{eq:diagonal} by
 \[\ld^{\frac{\eps_0}{2}}\|f\|_2\times\Bigl( \ld^{\frac12-\eps_0}\sup_\nu\|P^\nu_{\theta_0}g\|^2_2\Bigr)^\frac12,\]
which corresponds to \eqref{eq:main result2}.

{\bf The second microlocalization}.
For the off diagonal part \eqref{eq:off diagonal},
we will reduce the matters to
a bilinear oscillatory integrals as in \cite{Blair Sogge:refine}.
Fixing $\ell\geq \log M/\log 2$,
we see that if $2^{\ell}\leq|\nu-\upsilon|< 2^{\ell+1}$,
then the distance between $\gamma_{\nu}$ and
$\gamma_{\upsilon}$ in the sense of
Remark \ref{rmk} is approximately $2^\ell \theta_0$.
To explore this and use orthogonality argument,
one naturally employs wider tubes to
collect thinner tubes by making use of  the second mircolocalization.
Precisely, up to some negligible terms,
we may write for  $\theta_\ell=2^\ell\theta_0$
with $c_0$ to be specified later
\[\chi_\ld Q^{\nu}_{\theta_0}\,f(x)\approx
  \sum_{\sigma_1\in\Z^2}\bigl(\chi_\ld Q^{\sigma_1}_{c_0\theta_\ell}\bigr)\circ Q^{\nu}_{\theta_0}\,f(x),\;
  \chi_\mu P^{\upsilon}_{\theta_0}\,g(x)\approx
  \sum_{\sigma_2\in\Z^2}\bigl(\chi_\mu P^{\sigma_2}_{c_0\theta_\ell} \bigr)\circ P^{\upsilon}_{\theta_0}\,g(x).\]
Noting that the kernels of the operators $(\chi_\ld Q^{\sigma_1}_{c_0\theta_\ell})\circ Q^{\nu}_{\theta_0}$ and
$(\chi_\mu P^{\sigma_2}_{c_0\theta_\ell})\circ P^{\upsilon}_{\theta_0}$
decrease rapidly unless
$T_{C_1c_{0}\theta_\ell}(\gamma_{\sigma_1})\cap T_{C_1\theta_0}(\gamma_\nu)\neq\emptyset$
and  $T_{C_1c_{0}\theta_\ell}(\gamma_{\sigma_2})\cap T_{C_1\theta_0}(\gamma_\upsilon)\neq\emptyset$,
we have by choosing $M$ large enough,
there are $N_0=N_0(c_0,M)$ and $N_1$
such that up to some negligible terms
\begin{align}
  &\sum_{2^\ell\leq |\nu-\upsilon|< 2^{\ell+1}}\chi_\ld Q^{\nu}_{\theta_0}f(x)\;\chi_\mu P^{\upsilon}_{\theta_0}g(x)\label{eq:bi}\\
  =\sum_{\sigma_1,\,\sigma_2\,\in\,\Z^2,\, N_0\leq|\sigma_1-\sigma_2|\leq N_1}&\;\sum_{2^\ell\leq|\nu-\upsilon|< 2^{\ell+1}}
    \bigl(\chi_\ld Q^{\sigma_1}_{c_0\theta_\ell} \bigr)\circ Q^{\nu}_{\theta_0}f(x)\;
    \bigl(\chi_\mu P^{\sigma_2}_{c_0\theta_\ell}\bigr)\circ P^{\upsilon}_{\theta_0}g(x)\nonumber.
\end{align}
Moreover, we may find a $C_3>0$ having the property that
for every $\sigma_1$ and $\sigma_2$, there are
$\nu(\sigma_1)$ and $\upsilon(\sigma_2)$
such that $|\nu-\nu(\sigma_1)|,|\upsilon-\upsilon(\sigma_2)|\geq C_3 2^\ell$
implies
\[\bigl\|\bigl(\chi_\ld Q^{\sigma_1}_{c_0\theta_\ell}\bigr)\circ Q^{\nu}_{\theta_0}f\bigr\|_{L^\infty}
\lesssim_N \ld^{-N},\;
\bigl\|\bigl(\chi_\mu P^{\sigma_2}_{c_0\theta_\ell}\bigr)\circ P^{\upsilon}_{\theta_0}g\bigr\|_{L^\infty}\lesssim_N\mu^{-N}.\]
for all $N=1,2,\ldots$.
Therefore, we may estimate \eqref{eq:bi} as follows
\begin{align*}
    \Bigl\|\sum_{2^\ell\leq|\nu-\upsilon|< 2^{\ell+1}}
    &\chi_\ld Q^{\nu}_{\theta_0}f\;\chi_\mu P^{\upsilon}_{\theta_0}g\Bigr\|^2_2\\
    \lesssim\sum_{\substack{N_0\leq|\sigma_1-\sigma_2|\leq N_1\\|\sigma_1-\tilde\sigma_1|+|\sigma_2-\tilde\sigma_2|\leq C}}
    &\int  T^{\sigma_1,\,\sigma_2}_{\ld,\,\mu,\,\theta_\ell} F(x)\;
   \overline{T^{\tilde\sigma_1,\,\tilde\sigma_2}_{\ld,\,\mu,\,\theta_\ell} F(x)} \, dx \\
   &+\sum_{\substack{N_0\leq|\sigma_1-\sigma_2|\leq N_1\\|\sigma_1-\tilde\sigma_1|+|\sigma_2-\tilde\sigma_2|\geq C}}
    \int T^{\sigma_1,\,\sigma_2}_{\ld,\,\mu,\,\theta_\ell} F(x)\;
   \overline{T^{\tilde\sigma_1,\,\tilde\sigma_2}_{\ld,\,\mu,\,\theta_\ell} F(x) } \, dx,
\end{align*}
 where $N_0$ can be sufficiently large by choosing $c_0$ small and
 \begin{equation}
 \label{eq:bilinear osc-integral}
   T^{\sigma_1,\,\sigma_2}_{\ld,\,\mu,\,\theta_\ell} F(x)=\iint \bigl(\chi_\ld\circ Q^{\sigma_1}_{c_0\theta_\ell}\bigr)(x,y)\,
   \bigl( \chi_\mu\circ P^{\sigma_2}_{c_0\theta_\ell}\bigr)(x,z)\,F(y,z)\,dydz,
 \end{equation}
  \begin{equation}
   F(y,z)=\sum_{\substack{2^{\ell}\leq|\nu-\upsilon|<2^{\ell+1}
            \\|\nu(\sigma_1)-\nu|+|\upsilon(\sigma_2)-\upsilon|\leq C_{3}2^\ell}}\,
   Q^{\nu}_{\theta_0}\,f(y)\;P^{\upsilon}_{\theta_0}\,g(z),
 \end{equation}
 with $F(y,z)=0$ if $(y,z)\not\in B(y_0,C_0\delta)\times B(z_0, C_0\delta)$.
 It follows again from Lemma~\ref{lemma:orthogonality} that if we choose $C$ large enough, the second term in the expression preceding
 \eqref{eq:bilinear osc-integral} is negligible.
\medskip

To evaluate the first term there, we are reduced to estimating
 \begin{align}
   \sum_{N_0\leq |\sigma_1-\sigma_2|\leq N_1}\bigl\|T^{\sigma_1,\,\sigma_2}_{\ld,\,\mu,\,\theta_\ell}\,F\bigr\|^2_{L^2(B(0,\;\delta))}.
 \end{align}
 We shall need the following proposition whose proof is postponed to the next section.
 \begin{proposition}
   \label{prop:bilinear estimate}
Let
\begin{equation}
   T^{\sigma_1,\,\sigma_2}_{\ld,\,\mu,\,\theta} F(x)=\iint \bigl(\chi_\ld\circ Q^{\sigma_1}_{c_0\theta}\bigr)(x,y)\,
   \bigl( \chi_\mu\circ P^{\sigma_2}_{c_0\theta}\bigr)(x,z)\,F(y,z)\,dydz.
 \end{equation}
Assume as before that $\delta>0$ is sufficiently small and $\theta$ is larger than a fixed positive constant times $\theta_0$. Then if $N_0$ is suffciently large and $N_1>N_0$ is fixed, there exists a positive constant $C=C_{\eps_0}$ such that
   \begin{equation}
     \label{eq:bilinear osc-estimate}
     \bigl\|T^{\sigma_1,\,\sigma_2}_{\ld,\,\mu,\,\theta}F\bigr\|_{L^2(B(0,\delta))}\leq C\,\theta^{-1/2}\|F\|_2,\ \ \ {\rm\it if}\ N_0\le|\sigma_1-\sigma_2|\le N_1.
   \end{equation}
 \end{proposition}
Assuming \eqref{eq:bilinear osc-estimate}, we can now complete the proof of Theorem \ref{theorem:main}.
In fact, we have
 \begin{align*}
  & \Bigl\|\sum_{2^\ell\leq|\nu-\upsilon|< 2^{\ell+1}}
   \chi_\ld Q^{\nu}_{\theta_0}f\;\chi_\mu P^{\upsilon}_{\theta_0}g\Bigr\|^2_2\\
   \leq\,C\,(2^\ell\theta_0)^{-1}\sum_{N_0\leq|\sigma_1-\sigma_2|\leq N_1}&
            \iint\Bigl|\sum_{\substack{2^{\ell}\leq|\nu-\upsilon|<2^{\ell+1}
            \\|\nu(\sigma_1)-\nu|+|\upsilon(\sigma_2)-\upsilon|\leq C_{3}2^\ell}}
            Q^{\nu}_{\theta_0}\,f(y)\,
            P^{\upsilon}_{\theta_0}\,g(z)\,\Bigr|^2dydz.
 \end{align*}
 Notice that
 \begin{align*}
 &\iint\Bigl|\sum_{\substack{2^{\ell}\leq|\nu-\upsilon|<2^{\ell+1}
            \\|\nu(\sigma_1)-\nu|+|\upsilon(\sigma_2)-\upsilon|\leq C_{3}2^\ell}}
            Q^{\nu}_{\theta_0}\,f(y)\,
            P^{\upsilon}_{\theta_0}\,g(z)\,\Bigr|^2dydz\\
 =\sum_{\substack{2^{\ell}\leq|\nu-\upsilon|<2^{\ell+1}
            \\|\nu(\sigma_1)-\nu|+|\upsilon(\sigma_2)-\upsilon|\leq C_{3}2^\ell}}
           & \sum_{\substack{2^{\ell}\leq|\nu'-\upsilon'|<2^{\ell+1}
            \\|\nu(\sigma_1)-\nu'|+|\upsilon(\sigma_2)-\upsilon'|\leq C_{3}2^\ell}}
            \Bigl\langle\bigl( Q^{\nu'}_{\theta_{0}}\bigr)^{*}\circ Q^{\nu}_{\theta_{0}}\,f,\,f\Bigr\rangle \;
             \Bigl\langle \bigl(P^{\upsilon'}_{\theta_{0}}\bigr)^{*}\circ P^{\upsilon}_{\theta_{0}}\,g,\,g\Bigr\rangle,
 \end{align*}
where
$\|\bigl( Q^{\nu'}_{\theta_{0}}\bigr)^{*}\circ Q^{\nu}_{\theta_{0}}\|_{L^{2}\rightarrow L^{2}}=O(\lambda^{-N})$
and
$\|\bigl(P^{\upsilon'}_{\theta_{0}}\bigr)^{*}\circ P^{\upsilon}_{\theta_{0}}\|_{L^{2}\rightarrow L^{2}}= O(\mu^{-N})$
if $|\nu-\nu'|+|\upsilon-\upsilon'|\geq C$ for $C$ large.
Consequently, we have up to some negligible terms
\begin{align*}
\|\eqref{eq:bi}\|_{2}^{2}\leq &C\, (2^{\ell}\theta_{0})^{-1}\sum_{N_{0}\leq |\sigma_{1}-\sigma_{2}|\leq N_{1}}
     \sum_{|\nu-\nu(\sigma_{1})|+|\upsilon-\upsilon(\sigma_{2})|\leq C_{3}\,2^{\ell}}\|Q^{\nu}_{\theta_{0}}\,f\|_{2}^{2}
     \|P^{\upsilon}_{\theta_{0}}g\|_{2}^{2}\\
     \leq &C\,(2^{\ell}\theta_{0})^{-1}\biggl(\sup_{\sigma_{1}}\sum_{|\nu-\nu(\sigma_{1})|\leq C_{3}\,2^{\ell}}
     \|Q^{\nu}_{\theta_{0}}f\|_{2}^{2}\biggr)
     \cdot\biggl(\sum_{\sigma_{2}}\sum_{|\upsilon-\upsilon(\sigma_{2})|\leq C_{3}\,2^{\ell}}
     \|P^{\upsilon}_{\theta_{0}}g\|_{2}^{2}\biggr)\\
     \leq& C\,(2^{\ell}\theta_{0})^{-1}\,\|g\|_2^2\;\sup_{\nu\in\Z^2}\|Q^\nu_{2^\ell\theta_0}f\|_2^2.
\end{align*}
Thanks to the fact that we are allowed to have an extra small power of $\lambda$, we may sum over $1\lesssim\ell\lesssim \log \lambda$ to finish the proof of \eqref{eq:main result1}. To get \eqref{eq:main result2}, one notes that the above sum is again symmetric, thus we may interchange the role of $Q^{\nu}_{\theta_0}f$ and $P^{\upsilon}_{\theta_0}g$ to get 
\[\Big\|\sum_{2^\ell\leq |\nu-\upsilon|< 2^{\ell+1}}\chi_\ld Q^{\nu}_{\theta_0}f(x)\;\chi_\mu P^{\upsilon}_{\theta_0}g(x)\Big\|_{2}^{2}\leq C\,(2^{\ell}\theta_{0})^{-1}\,\|f\|_2^2\;\sup_{\nu\in\Z^2}\|P^\nu_{2^\ell\theta_0}g\|_2^2,\]
summing over $\ell$ finishes the proof of \eqref{eq:main result2}.
%
%

\section{Proof of the main theorem II:  bilinear oscillatory integral estimates}
In this section, we take $\theta=c_0\theta_\ell$ and prove Proposition \ref{prop:bilinear estimate}. We work in the geodesic normal coordinates about a fixed point 
$\tilde x\in T_{C\theta}(\gamma_{\sigma_1})\cap T_{C\theta}(\gamma_{\sigma_2})$. Without loss of generality, we may assume $\tilde x\in \gamma_{\sigma_1}$ and the geodesic $\gamma_{\sigma_1}$ is parameterized by $\{(0,s):|s|\leq 2\}$.
In the following, we denote by $\phi(x,y)=d_\g\bigl((x_1,x_2),(y_1,y_2)\bigr)$
the geodesic distance between $x$ and $y$.

In order to estimate the $L^2(B(0,\delta))$ norm of
\begin{equation} \label{eq:bilinear osc-integral'}   T^{\sigma_1,\,\sigma_2}_{\ld,\,\mu,\,\theta} F(x)=\iint \bigl(\chi_\ld\circ Q^{\sigma_1}_{\theta}\bigr)(x,y)\,
   \bigl( \chi_\mu\circ P^{\sigma_2}_{\theta}\bigr)(x,z)\,F(y,z)\,dydz,\end{equation}
we shall need the following lemma to further restrict the domain of $x,y,z$.
\begin{lemma} There exists a constant $C$, such that if we set $\Omega_1=T_{C\theta}(\gamma_{\sigma_1})$ and $\Omega_2=T_{C\theta}(\gamma_{\sigma_2})$we have
\[ \left\|\iint_{y\not\in\Omega_1} \bigl(\chi_\ld\circ Q^{\sigma_1}_{\theta}\bigr)(x,y)\,
   \bigl( \chi_\mu\, P^{\sigma_2}_{\theta}\bigr)(x,z)\,F(y,z)\,dydz\right\|_{L^2(B(0,\delta))}\le C_N \lambda^{-N}\|f\|_{2}\|g\|_{2},\]
and
\[ \bigl\|T^{\sigma_1,\,\sigma_2}_{\ld,\,\mu,\,\theta}F\bigr\|_{L^2(B(0,\delta)\setminus\Omega_1)}\le C_N \lambda^{-N}\|f\|_{2}\|g\|_{2}.\]
Similarly, we have
\[ \left\|\iint_{z\not\in\Omega_2} \bigl(\chi_\ld\circ Q^{\sigma_1}_{\theta}\bigr)(x,y)\,
   \bigl( \chi_\mu\, P^{\sigma_2}_{\theta}\bigr)(x,z)\,F(y,z)\,dydz\right\|_{L^2(B(0,\delta))}\le C_N \mu^{-N}\|f\|_{2}\|g\|_{2},\]
and
\[ \bigl\|T^{\sigma_1,\,\sigma_2}_{\ld,\,\mu,\,\theta}F\bigr\|_{L^2(B(0,\delta)\setminus\Omega_2)}\le C_N \mu^{-N}\|f\|_{2}\|g\|_{2}.\]\end{lemma}
\begin{proof}
Since we know there are at most $O(\ld^2)$ many terms in the sum
\[
   F(y,z)=\sum_{\nu}\sum_{\upsilon:|\upsilon-\nu|\in[2^\ell,2^{\ell+1})}\,
   Q^{\nu}_{\theta_0}\,f(y)\;P^{\upsilon}_{\theta_0}\,g(z),
\]
it suffices to show the $L^2(B(0,\delta))$ norm of
\[\iint \bigl(\chi_\ld\circ Q^{\sigma_1}_{ \theta}\bigr)(x,y)\,
   \bigl( \chi_\mu\circ P^{\sigma_2}_{ \theta}\bigr)(x,z)\,f(y)\,g(z)\,dydz
\]
satisfies our claim.

Indeed, by Lemma \ref{lemma:Kernels}, we can find $C$ such that if $x\not\in T_{C\theta}(\gamma_{\sigma_1})$ or $y\not\in T_{C\theta}(\gamma_{\sigma_1}),$ 
  \[ |(\chi_\ld Q^\nu_\theta)(x,y)|\leq C_N \ld^{-N}.\]
Thus
\[\left\|\int \bigl(\chi_\ld\circ Q^{\sigma_1}_{ \theta}\bigr)(x,y)\,
  f(y)\,dy\right\|_{L^\infty(dx)}\le C_N \ld^{-N}\|f\|_{L^2},
\]
while we know $\chi_\mu$ has $L^2\rightarrow L^2$ norm 1, so
\[\left\|\int \bigl(\chi_\mu\circ P^{\sigma_2}_{ \theta}\bigr)(x,y)\,
  g(z)\,dz\right\|_{L^2(dx)}\le \|g\|_{L^2}.
\]
Therefore
\[\left\|\iint \bigl(\chi_\ld\circ Q^{\sigma_1}_{ \theta}\bigr)(x,y)\,
   \bigl( \chi_\mu\circ P^{\sigma_2}_{ \theta}\bigr)(x,z)\,f(y)\,g(z)\,dydz\right\|_{L^2}\le C_N \lambda^{-N}\|f\|_{L^2}\|g\|_{L^2}
\]
as claimed.

The second part of our lemma follows from the exact same proof.
\end{proof}

\begin{figure}[ht]
 \begin{center}
\includegraphics[width=10cm]{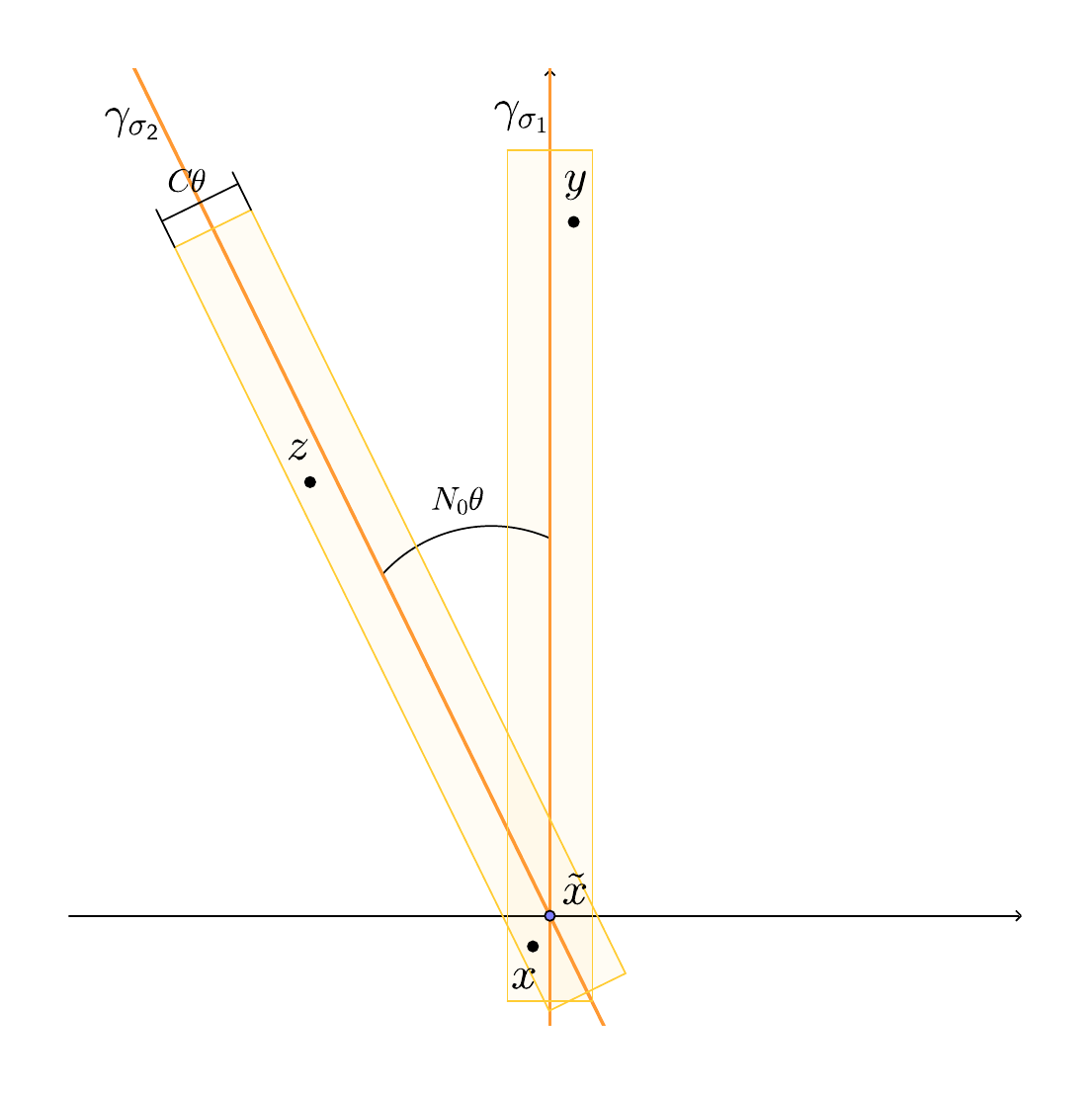}

\end{center}
\caption{}
\end{figure}
\begin{remark}\label{remark:Separation}By the above lemma, we see that we can assume in \eqref{eq:bilinear osc-integral'},
$y\in T_{C\theta}(\gamma_{\sigma_1})$,
$z\in T_{C\theta}(\gamma_{\sigma_2})$,
and $x\in T_{C\theta}(\gamma_{\sigma_1})\cap T_{C\theta}(\gamma_{\sigma_2})$.
Moreover,
if $N_0\leq |\sigma_1-\sigma_2|\leq N_1$, then
we may assume the angle $\text{Ang}(x;y,z)$ between the geodesic connecting $x$ and $y$ and the one connecting $x$ and $z$ belongs to $[\theta,\,\tilde C_4\theta]$. This geometric assumption yields 
$x,y,z\in T_{C_4\theta}(\gamma_{\sigma_1})$ for some large constant $C_4$. Moreover, we also have $\angle(\gamma_{\sigma_1},\gamma_{\sigma_2})\ge N_0\theta.$ Noticing that $d_g(x,y)$ and $d_g(x,z)$ are comparable to 1, we claim that for $N_0$ sufficiently large, we can find $c>0$ such that 
\begin{equation}\label{eq:1st separation}|y_1-z_1|>c\theta.\end{equation}
Indeed, it is easy to see that $|y_1|\le C\theta$ and $d_g(z,\gamma_{\sigma_2})\le C\theta$. Since the constant $C$ here is a uniform constant as in Lemma \ref{lemma:Kernels}, we can choose $N_0\gg C$. Then we have $|z_1|\ge N_0\theta-C\theta$, see Figure 1. Therefore $|y_1-z_1|\ge N_0\theta-2C\theta\ge c\theta$ as claimed.
\end{remark}
\medskip

Returning to $T^{\sigma_1,\sigma_2}_{\ld,\,\mu,\,\theta}F(x)$,
we have from Cauchy-Schwarz
\begin{multline*}
\bigl\|T^{\sigma_1,\sigma_2}_{\ld,\mu,\theta}F\bigr\|^2_2
\\
\lesssim\ld\,\mu\,\iint
\Bigl| \int e^{i\mu\Phi_\epsilon(x;\,(y_1,y_2),\,(z_1,z_2))}
a_{\lambda,\,\mu,\,\theta}^{\sigma_1,\sigma_2}(x,y,z)\,
F(y,\,z)dy_1dz_1\Bigr|^2dxdy_2dz_2,
\end{multline*}
where $\epsilon=\ld/\mu$ and
\begin{align*}
  \Phi_\epsilon(x;y,z)=&\epsilon\phi(x,y)+\phi(x,z),\\
a_{\ld,\,\mu,\,\theta}^{\sigma_1,\,\sigma_2}(x;y,z)=&a_{\sigma_1,\,
\theta}(x,y)\;b_{\sigma_2,\,\theta}(x,z).
\end{align*}
Fix $y_2$ and $z_2$, it suffices to prove
\begin{multline}\label{eq:c-s ed}
\int_{\mathbb R^2}
\Bigl| \int_{\R^{2}} e^{i\mu\Phi_\epsilon(x;\,(y_1,y_2),\,(z_1,z_2))}
a_{\lambda,\,\mu,\,\theta}^{\sigma_1,\sigma_2}(x,y,z)\,G(y_1,z_1)\,dy_1dz_1\Bigr|^2dx
\\
\leq\,C\, (\ld\mu\theta)^{-1}\|G\|_{L^2}^2,
\end{multline}
uniformly with respect to $y_2,z_2$ where we set $G(y_1,z_1)=F(y,z)$ for brevity.
\medskip

Squaring the left side of \eqref{eq:c-s ed} shows that we need to estimate
\begin{align}
  \iint e^{i\,\mu\, \Psi(x;\,y_1,\,y_1',\,z_1,\,z_1')}
  A_{\ld,\,\mu,\,\theta}^{\sigma_1,\,\sigma_2}(x;\,y_1,\,y_1',\,z_1,\,z_1')\,
    G(y_1,z_1)\overline{G(y'_1,z'_1)}\,dxdy_1dz_1dy'_1dz'_1,\label{eq:off diag}
\end{align}
where
\begin{align*}
  &A_{\ld,\,\mu,\,\theta}^{\sigma_1,\,\sigma_2}(x;\,y_1,\,y_1',\,z_1,\,z_1')
  =a_{\ld,\,\mu,\,\theta}^{\sigma_1,\,\sigma_2}(x,\,(y_1,y_{2}),\,(z_1,z_{2}))\,
  \overline{a_{\ld,\,\mu,\,\theta}^{\sigma_1,\,\sigma_2}(x,\,(y'_1,y_{2}),\,(z'_1,z_{2}))},\\
 & \Psi=\Psi_{\epsilon,\,y_2,\,z_2}(x;\,y_1,y_1',z_1,z_1')
  =\Phi_\epsilon(x;\,(y_1,y_2),(z_1,z_2))-\Phi_\epsilon(x;\,(y'_1,y_2),(z'_1,z_2)).
\end{align*}
Set
\beq\label{eq:osckernel}
K_{\ld,\,\mu,\,\theta}^{\sigma_1,\,\sigma_2}(y_1,y_1';z_1,z_1')
=\int_{\R^2}e^{i\,\mu \,\Psi(x;\,y_1,y_1',z_1,z_1')}A_{\ld,\,\mu,\,\theta}^{\sigma_1,\,\sigma_2}(x;\, y_1,\,y_1',\,z_1,\,z_1')\,dx.
\eeq
Then by Schur test, we are reduced to proving
\begin{multline*}
\sup_{y_1',\,z'_1}\int_{\R^2}
\bigl|K_{\ld,\,\mu,\,\theta}^{\sigma_1,\,\sigma_2}(y_1,y_1';z_1,z_1')\bigr|\,dy_1dz_1,\;
\sup_{y_1,\,z_1}\int_{\R^2}
\bigl|K_{\ld,\,\mu,\,\theta}^{\sigma_1,\,\sigma_2}(y_1,y_1';z_1,z_1')\bigr|\,dy_1'dz_1'
\\
\leq C/\ld\mu\theta.
\end{multline*}
By symmetry, we shall only deal with the first one.

By Remark \ref{remark:Separation}, we have

\beq\label{eq:t-region}
|y_1-z_1|\ge c\,\theta,\,|y'_1-z'_1|\ge c\,\theta.
\eeq
This would allow us to study the oscillatory integral \eqref{eq:osckernel}
using the strategy of \cite{Sogge:Tohoku} and a
change of  variables argument similar to the one in p. 217-218 of \cite{BurqGerardTzvetkov:Invention}.
In fact, if we let $\psi(x,y_1)=\phi(x,(y_1,y_2))$, then $\psi$ is a Carleson-Sj\"olin phase for fixed $y_2$, i.e.
\begin{equation}\label{eq:Carlesion-Sjolin}
{\rm det}
  \begin{pmatrix}
      \psi''_{x_1y_1} & \psi''_{x_2y_1} \\
      \psi'''_{x_1y_1y_1} & \psi'''_{x_2y_1y_1} \\
   \end{pmatrix}
 \neq 0,
\end{equation}
see \cite{Sogge:93book, Sogge:Tohoku}.
Changing variables  $(y_1,z_1)\mapsto (\tau,\tau')$, $(y_1',z_1')\mapsto (\titau,\titau')$ by
$$
\begin{cases}
\tau=&\frac{\ld}{2\mu}(y_1-z_1)^2\\
\tau'=&z_1+\frac\ld\mu y_1
\end{cases}
,\;
\begin{cases}
\titau=&\frac{\ld}{2\mu}(y_1'-z_1')^2\\
\titau'=&z_1'+\frac\ld\mu y'_1
\end{cases},
$$
where we may assume $y_{1}>z_{1}$ by symmetry. It is clear that
 the above bijective mapping sends variables
from $\{y_{1}-z_{1}\ge c\theta\}$ to
$\{(\tau,\tau'):\tau\ge c\lambda\theta^{2}/2\mu\}$,
whose Jacobian reads
\[\frac{D(\tau,\,\tau')}{D(y_{1},\,z_{1})}=(1+\epsilon)(2\epsilon\tau)^{1/2}.\]
The phase function in \eqref{eq:osckernel} goes to
\[\widetilde\Psi (x;\, \tau,\,\tilde\tau,\,\tau',\,\tilde\tau')=\Psi(x;\,y_{1},\,y_{1}',\,z_{1},\,z_{1}'),\]
under the change of variables. The Carleson-Sj\"olin condition allows us to
obtain as in  \cite{BurqGerardTzvetkov:Invention}
\[|\nabla_{x}\widetilde\Psi(x;\,\tau,\,\tilde\tau,\,\tau',\,\tilde\tau')|\approx |\tau-\tilde\tau|+|\tau'-\tilde\tau'|,\]
\[|\partial^{\alpha}_{x}\widetilde\Psi(x;\,\tau,\,\tilde\tau,\,\tau',\,\tilde\tau')|\le\,C_{\alpha}
( |\tau-\tilde\tau|+|\tau'-\tilde\tau'|),\,|\alpha|\leq 5.\]
In view of integration by parts and relation \eqref{eq:t-region},
we have  for fixed $(y_1',z_1')$ hence fixed $(\titau,\titau')$, thus
\begin{align*}
&\iint_{y_1-z_1\ge c\theta}\bigl|K_{\ld,\,\mu,\,\theta}^{\sigma_1,\,\sigma_2}(y_1,y_1';z_1,z_1')\bigr|\;dy_1dz_1\\
\le&C\iint_{\tau\ge c\ld\theta^2/2\mu}(1+\mu|\tau-\titau|+\mu |\tau'-\titau'|)^{-5}\;\Bigl(\frac\ld\mu \tau\Bigr)^{-1/2}\;d\tau d\tau'\\
\le& C\,/\ld\mu\theta,
\end{align*}
finishes the proof.

\begin{remark}
As mentioned before, it would be interesting to see that if one could get rid of the $\eps$-loss  that appears in Theorem \ref{theorem:main}. In the earlier work \cite{Sogge:Tohoku} of the second author on the linear case,  there is no $\eps$-loss in his result, while the power of $\vertiii{e_\ld}_{KN(\ld)}$ is less favorable. Thus it is natural to consider that if one could apply the strategies presented in \cite{Sogge:Tohoku} to get a $\eps$-loss free version of Theorem \ref{theorem:main}. However, 
it seems more difficult to use the microlocal decomposition if we want to get rid of the $\eps$. Without the help of microlocal techniques, the separation of $(y_0,\delta)$ and $(z_0,\delta)$ in the radial direction becomes problematic, and it seems difficult to get around by simply applying ideas in \cite{Sogge:Tohoku}.
\end{remark}

\end{document}